\documentclass[11pt]{article}
\usepackage[latin1]{inputenc}
\usepackage[english]{babel}
\usepackage{amssymb,latexsym}
\usepackage{amsmath,amsthm,amstext,amsfonts}
\usepackage[normalem]{ulem}
\usepackage{cite}

\usepackage{graphicx}
\graphicspath{{./figuras}}

\usepackage{auto-pst-pdf}
\usepackage{pstricks}
\usepackage{pst-plot,pst-node,pstricks-add}

\usepackage{graphicx}
\graphicspath{{./figuras}}

\usepackage{verbatim,fancyvrb}

\usepackage{physics}

\newtheorem{defn}{Definition}[section]
\newtheorem{thm}{Theorem}[section]
\newtheorem{prop}{Proposition}[section]
\newtheorem{cor}{Corollary}[section]
\newtheorem{lem}{Lemma}[section]
\newtheorem{rmk}{Remark}[section]
\newtheorem{exmp}{Example}[section]

\makeatletter

\let \al=\alpha
\let \be=\beta

\let \vare=\varepsilon

\let \de=\delta
\let \th=\theta

\let \ga=\gamma
\let \p=\partial
\let \q=\quad
\let \qq=\qquad
\let \med=\medskip
\let \smal=\smallskip

\let \ol=\overline

\newcommand{\R}{\mathbb{R}}
\newcommand{\N}{\mathbb{N}}

\begin{document}


\vspace{0.2cm}

\begin{center}
\textbf{\Large{Asymptotic behaviour of  general nonautonomous Nicholson equations with mixed monotonicities}}
	\end{center}

\centerline{\scshape Teresa Faria
\footnote{This work was  supported by the National Funding from FCT - Funda\c c\~ao para a Ci\^encia e a Tecnologia (Portugal) under project UIDB/04561/2020.}}
\smallskip
{\footnotesize
 \centerline{Departamento de Matem\'atica and CMAFcIO,}
   \centerline{ Faculdade de Ci\^encias, Universidade de Lisboa, Campo Grande, 1749-016 Lisboa, Portugal}}

%
%
%
\vskip .5cm

\begin{abstract}  

A general nonautonomous Nicholson equation with multiple pairs of delays in {\it mixed monotone} nonlinear terms is studied.
Sufficient conditions for permanence are given, with explicit lower and upper uniform bounds  for all positive solutions. Imposing  an additional condition on the size of some of the delays, and by using an adequate difference equation of the form $x_{n+1}=h(x_n)$, we show that  all positive solutions 
 are globally attractive. In the case of a periodic equation, a criterion for existence of a globally attractive positive solution is provided.
  The results here constitute a  significant improvement  of  recent literature, in view of   the generality of the equation under study and of sharper  criteria  obtained  for situations covered in recent works. 
  Several examples illustrate the results.
\end{abstract}

 


\section{Introduction}
\setcounter{equation}{0}


Recently, delay differential equations (DDEs) with  one or several  {\it mixed monotone}  nonlinear terms have received increasing attention 
  \cite{bb2016,bb2017,ER19,ElRuiz20,ElRuiz22,FO19,GHM,huang20,smith08}. A variant of the famous Nicholson's blowflies equation
\begin{equation}\label{N}
    x'(t)=- \delta x(t)+ p x(t-\tau)e^{-a x(t-\tau)} \q (p,\de,\tau>0),
  \end{equation}
   introduced in \cite{GBN} to model  the life cycle of the Australian  blowfly,
is the autonomous Nicholson-type equation with two different delays studied  in \cite{ER19}:
\begin{equation} \label{Nich_ER}
  x'(t)=- \delta x(t)+ p x(t-\tau)e^{-a x(t-\sigma)},
  \end{equation}
  with  $p,\de,a,\tau,\sigma> 0$, where $\tau,\sigma$   represent respectively the  incubation and the maturation delays for the blowfly \cite{bb2017}; a biological interpretation of the other parameters can be found in  e.g.\cite{ER19,GBN,LongGong,Xu2}.   In  \eqref{Nich_ER}, 
  there is  a {\it  mixed monotonicity} in the nonlinear term, which is expressed  by  $g(x(t-\tau),x(t-\sigma))$ with  $g(x,y):=pxe^{-ay}$   monotone increasing in the first variable and monotone decreasing in the second one. The global attractivity of its positive equilibrium $K=\frac{1}{a}\log (p/\de)$  was  proven by  El-Morshedy and  Ruiz-Herrera \cite{ER19} under the conditions 
  $p>\delta$,  $\tau\ge\sigma>0$ and
  $(e^{\de \sigma}-1)\log(p/\de)\le 1.$    Note however that for \eqref{Nich_ER} with $p/\de >1$,  $K$ is never  globally asymptotically stable (GAS) for all values of the delays $\tau,\sigma$, as shown by an example in \cite{Faria22} -- whereas for the classic Nicholson's equation  \eqref{N} the  equilibrium $K>0$ is {\it absolutely} (i.e., for all $\tau>0$) GAS if $1<p/\de\le e^2$.

%

In \cite{FariaPrates,LongGong}, the  authors studied the mixed monotone Nicholson's equations with multiple pairs of time-varying delays given by
\begin{equation}\label{eq:nicholson}
    x'(t)=\beta (t) \bigg(-\delta x(t)+\sum_{j=1}^{m} p_j  x(t-\tau_j(t)) e^ {-a_j x(t-\sigma_j(t))}\bigg)\, , \q t\ge t_0, 
\end{equation}
where $p_j , a, \delta \in (0, \infty )$, $\beta(t), \sigma_j(t), \tau_j(t)$ are continuous, nonnegative and bounded, and $\be(t)$  is also bounded from below  by a positive constant; see also \cite{bb2017} for the study of a Mackey-Glass equation with one pair of {\it mixed} delays and for the  open problem about the stability of equilibria for \eqref{eq:nicholson} with $m=1$.
  When $p:=\sum_{j=1}^{m} p_j  \le\delta$, Long and Gong \cite{LongGong} showed that the equilibrium $0$  is a global attractor of all positive solutions. 
If  $p>\de$, 
 there is always a positive equilibrium $K$, often called the {\it carrying capacity},  defined by the identity
\begin{equation}\label{equilK} 
 \sum_{j=1}^m p_j  e^{-a_j K}=\delta.
\end{equation}
Using techniques inspired in  \cite{ER19},  Faria and Prates  \cite{FariaPrates}   established a criterion for the global attractivity of the positive equilibrium  $K$ of \eqref{eq:nicholson}, which retrieves  the one in  \cite{ER19} for the very particular case of equation \eqref{Nich_ER} and in fact solves (in a broader setting) the open problem set in \cite{bb2017} .

The main goal of this paper is to study the   general nonautonomous scalar model of the form
 \begin{equation}\label{general}
    x'(t)=\sum_{j=1}^{m} p_j(t)  x(t-\tau_j(t)) e^ {-a_j(t) x(t-\sigma_j(t))}-\delta(t) x(t),\qq t\ge t_0,
\end{equation}
with coefficients and delays continuous and nonnegative (but with coefficients  not necessarily bounded), and provide criteria for permanence and global attractivity of  all solutions of \eqref{general} which generalize and enhance previous results in the literature. As a sub-product,  for a periodic equation \eqref{general}, a criterion for the existence of a globally attractive positive solution is given.

The main novelties of the present framework are 
 to consider the very general equation \eqref{general}, where all the coefficients and delays  are non-autonomous
and  the boundedness of  the coefficients $p_j(t)$ and $\de(t)$ is not a priori  assumed.  
Moreover, 
   the  results in the present paper also improve the main  criterion in  \cite{FariaPrates} on the global attractivity of $K$ given by \eqref{equilK}, as 
 several   hypotheses  in \cite{FariaPrates} are proven here to be superfluous. 

Here,   the   technique in   \cite{bb2016} is  a starting point to find conditions for the permanence,   whereas 
  the method to address the global attractivity of \eqref{general} is a follow up of the one in \cite{ER19,FariaPrates}.   
  The arguments in  \cite{bb2016,ER19,FariaPrates} are most helpful, however not only must careful adjustments
 be implemented, but also further reasonings and estimates provided in order to deal with the main difficulties one should overcome:  
 the nonautonomous character of all the coefficients and delays and the absence, in general, of a positive constant solution.


In the literature, some classical techniques to show that a solution of a nonautonomous DDE is a global attractor encompasse  construction of Lyapunov functionals or comparison results from  the theory of monotone systems \cite{smith}.  These methods are not easy to apply to DDEs with  mixed monotonicities. As mentioned, 
we extend
  the approach applied  in \cite{ER19} to \eqref{Nich_ER} and in \cite{FariaPrates} to  \eqref{eq:nicholson}. To show the global attractivity of any positive solution $x^*(t)$, the major difficulty is to rule out the existence of  solutions with an oscillatory behaviour about $x^*(t)$: for this,   the main tool  is to associate to the DDE \eqref{general} a suitable  autonomous difference equation  (DE) of the form $x_{n+1}=h(x_n)$,  where $h$ has negative   Schwarzian derivative.  Under certain assumptions,     the global attractivity of a fixed solution  for the original DDE is linked to  the  attractivity of  a fixed point for an auxiliary difference equation  $x_{n+1}=h(x_n)$.    For early works exploiting the relation between global attractivity for  scalar DDEs and  associated DEs, see e.g.  \cite{LizTT,M-PN1,M-PN2,smith08} and references therein.

To apply this approach, an auxiliary provisional hypothesis of existence of a positive equilibrium $K$  is first imposed, leading to the construction of
an original function $h$  which allows to generalize and improve the results in \cite{FariaPrates}. Then, this hypothesis is removed and we show how to carry out the arguments when such a constant solution does not exist, in order to derive the attractivity of any solution of \eqref{general}.
Of course, in the absence of a positive equilibrium  the results are not as sharp as the ones obtained when assuming its existence, as for \eqref{eq:nicholson}.  
For a related recent work dealing with Nicholson as well as other families of scalar DDEs with one pair of ``mixed monotone delays", see \cite{ElRuiz22}. 

We now recall briefly some notation and
standard definitions used in the next sections. Consider a scalar DDE with finite delays in $[0,\tau]$ ($\tau>0$) written in abstract form as  $x'(t)=f(t,x_t)$, where $f:D\subset \R\times C\to \R$ is continuous, $C:=C([-\tau, 0];\R)$ equipped with  the supremum norm  $\|\phi\|=\max_{\th\in[-\tau,0]}|\phi(\th)|$ is taken as the phase space and $x_t\in C$ is the segment of the solution $x(t)$ defined by $x_t(\th)=x(t+\th)$ for $\th\in [-\tau, 0]$.
As the present equations  are inspired in mathematical  biology models,  only  positive or nonnegative solutions  are  to be considered. In this way, we take $C_0^+:=\big\{\phi\in C : \phi(\th)\ge 0\ {\rm on}\ [-\tau,0),  \phi(0)>0 \}$  as the set of admissible initial conditions.
 Under uniqueness conditions, $x(t; t_0, \phi)\in \R$ denotes the solution of $x'(t)=f(t,x_t)$ with the initial condition 
\begin{equation}\label{IC}
    x_{t_0}=\phi\in C_0^+,
\end{equation}
for $(t_0,\phi)\in D$.
For \eqref{general},  standard results  from the theory of DDEs \cite{HL,smith,smith1} 
yield that solutions  with initial conditions \eqref{IC} are  defined and positive on $[t_0,\infty)$, see Lemma \ref{lem2.0}. The usual concepts of stability  and attractivity always refer to solutions with initial conditions in $C_0^+$, as given below.

\begin{defn} 
A solution  $x^*(t)$ of \eqref{general}-\eqref{IC} is said to be {\it persistent}, respectively {\it permanent}, if $\liminf_{t\to\infty}  x^*(t)>0$, respectively $0<\liminf_{t\to\infty}  x^*(t)\le \limsup_{t\to\infty}  x^*(t)<\infty$. 
The equation \eqref{general}  is {\bf  (uniformly) persistent}   if there exists $m>0$ such that
$$\liminf_{t\to\infty}  x(t;t_0, \phi)  \ge m\q {\rm for\  all}\ \phi \in C_0^+,$$
and   is   {\bf (uniformly)  permanent} if  there exist $ m,M>0$ such that
\begin{equation}\label{permanence}m\le \liminf_{t\to \infty} x(t;t_0, \phi)\le \limsup_{t\to \infty} x(t;t_0, \phi)\le M\q {\rm for\  all}\ \phi \in C_0^+.\end{equation}
A nonnegative
 solution $x^*(t)$ is called a {\bf global attractor}   if 
$\lim_{t\to\infty}(x(t)-x^*(t))=0$
for all solutions $x(t)=x(t;t_0,\phi)$  of \eqref{general}  with initial conditions \eqref{IC}. In this situation, \eqref{general} is said to be {\bf globally attractive} (in $C_0^+$) as every  solution with initial conditions in $C_0^+$ is a   global attractor.
\end{defn}

The organization of the paper is now presented. Section 2 is essentially  devoted to find conditions for the permanence of \eqref{general}. Futhermore,
 explicit positive lower and upper uniform bounds for solutions as $t\to\infty$   are  given.
As we shall see, the permanence is also a key  step to show  the attractivity of all solutions of \eqref{general}. Assuming an auxiliary condition, that will be latter removed, some preliminary lemmas about the asymptotic behaviour of solutions are  included in Section 3.
In Section 4, the main criteria for the global attractivity of a positive solution $x^*(t)$ are established
 by constructing an original auxiliary difference equation.
  Section 4 also includes a further result for periodic equations and several
 examples illustrating the application of our results.  The paper ends with a brief section of conclusions and discussion.

\section{Permanence}
\setcounter{equation}{0}

%
%


In this section, we establish the permanence  of \eqref{general}.  Consider the general nonautonomous scalar model \eqref{general}
under the following general conditions:
\begin{itemize}
\item[(A0)]
$p_j , a_j, \delta: [t_0, \infty )\to (0, \infty )$ and  $ \sigma_j, \tau_j:[t_0, \infty )\to [0, \infty )$ are  continuous functions with $ \sigma_j(t), \tau_j(t)$ bounded and $a_j(t)$ bounded from below and from above by positive constants ($1\le j\le m$).
\end{itemize} 
Note that $\delta(t)$ and $p_j(t)$ need not be bounded or bounded  from below by a positive constant.
For $$\tau=\max\{\sup_{t\ge t_0} \tau_j(t), \sup_{t\ge t_0} \sigma_j(t):j=1,\dots,m\},$$
 \eqref{general} is considered as a DDE in the phase space $C=C([-\tau,0];\R)$, with $C_0^+$ as the set of initial conditions. 
 For simplicity and without loss of generality,  we take $t_0=0$. 

For a bounded function $f:[0,\infty)\to \R$, we shall use the notation
$$f^+=\sup_{t\ge 0} f(x),\q f^-=\inf_{t\ge 0} f(x)$$
Thus, we have $0<a_j^-\le a_j(t)\le a_j^+ \, (1\le j\le m)$
 and $\tau=\max\{\tau_j^+,\sigma_j^+:j=1,\dots,m\}$. Set also
\[\begin{split}a^-=\min_j a_j^-,&\q  a^+=\max_j a_j^+,\\
 \sigma^+=\max_j \sigma_j^+,&\q  \tau^+=\max_j \tau_j^+.
\end{split}\]

\begin{lem}\label{lem2.0} Under (A0), solutions of \eqref{general} with initial conditions $x_{0}=\phi\in C_0^+$ are positive and defined on $[0,\infty)$.
\end{lem}

\begin{proof}
Since positive solutions $x(t)$ of \eqref{general}-\eqref{IC} satisfy 
$$-\de(t)x(t)\le x'(t)\le -\de(t)x(t)+\sum_{j=1}^{m} p_j(t)  x(t-\tau_j(t)) ,$$
the result follows by results on continuation and comparison of solutions \cite{HL,smith1}.
\end{proof}

Define $$p(t):=\sum_{j=1}^{m} p_j(t),\qq t\ge 0.$$

Assuming (A0) and 
 $\int_0^\infty\de(s)\, ds=\infty$, the following  criterion for  extinction of  \eqref{general} was established in \cite{XuCaoGuo}:
  (i) if $ \sup_{t\ge T} \frac{p(t)}{\de (t)}\le 1$, the equilibrium zero is globally asymptotically stable;
  (ii) if $ \sup_{t\ge T} \frac{p(t)}{\de (t)}< 1$, the equilibrium zero is globally exponentially stable.
See also
  \cite{LongGong}, where this result on extinction was proven
 only for \eqref{eq:nicholson} and assuming the more restrictive setting of $0<\be^-\le \be(t)\le \be^+<\infty$.

In this paper, we  analize the case $\inf_{t\ge T}  \frac{p(t)}{\de(t)}>1$ (for some $T>0$), and derive conditions  for permanence and global attractivity of positive solutions.
We however remark that the main result on attractivity, given in Section 4, only  requires the permanence of  each solution, rather than the uniform permanence. 

The main  assumptions to be imposed hereafter 
are described below:
\begin{itemize}
\item[(A1)] $ \liminf_{t\to\infty} \frac{p(t)}{\de(t)}>1$ and $ \limsup_{t\to\infty} \frac{p(t)}{\de(t)}<\infty$;
\item[(A2)] $ \limsup_{t\to \infty} \int_{t-\tau}^t p(s)\, ds<\infty, \ j=1,\dots,m$;
 \item[(A3)] $\int_0^\infty \de(s)\, ds=\infty$.
 \end{itemize}
 
 The analysis starts with a preliminary result. 

\begin{lem} \label{lem2.1}  Assume (A0), (A1).
  Then, there is no positive solution $x(t)$ of \eqref{eq:nicholson} such that $ \lim_{t\to \infty}x(t)=0$ or $ \lim_{t\to \infty}x(t)=\infty.$ 
\end{lem}

\begin{proof} 
Define $h_j(t,u,v)=p_j(t)ue^{-a_j(t)v},\, 1\le j\le m,$ and $h(t,u,v)=\sum_jh_j(t,u,v)$. Note that the functions $h_j(t,u,v)$ are monotone increasing in $u$ and monotone decreasing in $v$. 

From (A1),  for some $t\ge t^*$ set
$\al_0:=\inf_{t\ge t^*} \frac{p(t)}{\de(t)}>1$ and $ \ga_0:=\sup_{t\ge t^*} \frac{p(t)}{\de(t)}$. For $u>0, t\ge t^*$, we have
$$ \frac{h(t,u,m)}{\de(t)u}= \frac{\sum_jp_j(t)e^{-a_j(t)m}}{\de(t)}
\ge \frac{\sum_jp_j(t)e^{-a_j^+m}}{\de(t)}
\ge e^{-a^+m}\al_0\to \al_0>1\q {\rm as}\q m\to 0$$
and 
$$ \frac{h(t,u,M)}{\de(t)u}= \frac{\sum_jp_j(t)e^{-a_j(t)M}}{\de(t)}
\le \frac{\sum_jp_j(t)e^{-a_j^-M}}{\de(t)}\le e^{-a^-M}\ga_0\to 0\q {\rm as}\q M\to \infty.$$
Hence, there are  $\al_1,\ga_1$ and $m_0,M_0>0$ such that \begin{equation}\label{mM}
 \frac{h(t,u,m)}{\de(t)u}\ge \al_1>1\, ,\q \frac{h(t,u,M)}{\de(t)u}\le \ga_1 <1 \ ,\q t\ge t^*,
    \end{equation}
 for $ 0<m\le m_0$ and $M\ge M_0$.  From \cite[Theorems 4.1 and 5.2]{bb2016}, this implies that there is no solution $x(t)$ of \eqref{general}-\eqref{IC} with either $\lim_{t\to \infty} x(t)=0$ or 
 $\lim_{t\to \infty} x(t)=\infty$.
\end{proof}

\begin{thm} \label{thm2.1}  Assume (A0), (A1) and (A2). Then, any positive solution of  \eqref{general} is permanent. 
If in addition (A3)  holds,
then \eqref{general} is permanent, with  
\begin{equation}\label{permGen}
    m \le \liminf_{t\to \infty}x(t)\le \limsup_{t\to \infty}x(t) \le M
\end{equation}
 for any solution $x(t)$ of \eqref{general} with initial condition in $C_0^+$,
 where
\begin{equation}\label{boundsmM}
m= \frac{1}{a^+}\log \al\exp(-(2D +P)),\q M=\frac{1}{a^-}\log \ga\exp(2(D+P)).
\end{equation}
and \begin{equation}\label{alphagamma}
\begin{split}
&\al=\liminf_{t\to\infty} \frac{p(t)}{\de(t)},\q \ga=\limsup_{t\to\infty} \frac{p(t)}{\de(t)}\\
&D=\limsup_{t\to\infty} \int_{t-\tau}^t \de (s)\, ds,\q P=\limsup_{t\to\infty} \int_{t-\tau}^t p(s)\, ds \end{split}
\end{equation}

\end{thm}
\begin{proof} Define $\al,\ga, D, P$ as in \eqref{alphagamma}. Note that these constants  are well-defined and $D<P$ in virtue of (A1) and (A2).
As in the proof of Lemma \ref{lem2.1}, there are $t^*\ge 0$, with $\al_0:=\inf_{t\ge t^*} \frac{p(t)}{\de(t)}>1, \ga_0:=\sup_{t\ge t^*} \frac{p(t)}{\de(t)}$,  and $0<m_0\le M_0$ such that \eqref{mM} are satisfied  for $ 0<m\le m_0$ and $M\ge M_0$. 
 The proof  follows now in several steps. For Steps 1 and 2, we were inspired in  Berenzansky and Braverman \cite{bb2016}.

\smal

%

 \smal
{\it Step 1}. 
    Consider any positive solution $x(t)$. 
  In order to prove its permanence, some a priori estimates are established by  adapting some arguments in \cite[Theorem 5.6]{bb2016}.

 For any $\vare >0$ fixed, we may suppose that $t^*$ as above  was chosen so that the inequalities
 \begin{equation}\label{CC}
 C_1:=
 \sup_{t\ge t^*} \int_{t-\tau}^t \de (s)\, ds<D+\vare,\q C_2:=\sup_{t\ge t^*} \int_{t-\tau}^t p(s)\, ds<P+\vare.
     \end{equation}
  also hold. Define $t_0=t^*, t_j=t^*+j\tau, I_j=[t_{j-1},t_j]$ and
$$m_j=\min_{I_j}x(t),\q M_j=\max_{I_j}x(t),\q j\in\N.$$
For each $j\ge 2$, set $t_{j-1}^*\in I_{j-1}$ such that $x(t_{j-1}^*)=M_{j-1}$. Since $x'(t)\ge -\de(t) x(t)$,  we obtain
$$x(t)\ge  x(t_{j-1}^*)e^{- \int_{t_{j-1}^*}^t\de (s)\, ds},\q t\ge t_{j-1}^*,$$
thus 
\begin{equation}\label{m_j}
m_j\ge M_{j-1}e^{-2 C_1},\q t\in I_j.
\end{equation}
Write again  $h_k(t, u, v)=p_k (t) u e^{-a_k(t) v},\, 1\le k\le m$.
Since $x(t_{j-1})\le M_{j-1}$ and $x'(t)\le \sum_kh_k(t,x(t-\tau_k(t)),x(t-\sigma_k(t)))\le\sum_kp_k(t)x(t-\tau_k(t))$,  we have
\begin{equation}\label{x'}
 x'(t)\le p(t)\max \{M_{j-1}, \max_{s\in [t_{j-1},t]}x(s)\},\qq  t\ge t_{j-1}.
    \end{equation}
    Now,  for  $t\ge t_{j-1}$, we compare $x(t)$ with the solution of the ODE $y'(t)=p(t)y(t)$ with initial condition $y(t_{j-1})=M_{j-1}$, given by $y(t)=M_{j-1}e^{\int_{t_{j-1}}^t p(s)\, ds}$. From \eqref{CC} and \eqref{x'}, it follows
    $$x(t)\le y(t)\le M_{j-1}e^{C_2},\q t\in I_j,$$
    leading to
    \begin{equation}\label{M_j}
M_j\le M_{j-1}e^{C_2}
\end{equation}
and
\begin{equation}\label{x}x(t)\le M_{j-1}e^{2C_2},\qq t\in I_{j+1}.\end{equation}

 Step 2: The permanence of the solution $x(t)$ is now established.
   
       Choose any $m,M>0$ such that conditions \eqref{mM} are satisfied. For the sake of contradiction, assume that $x(t)$ is not bounded. Thus, for any $\ol M>Me^{2(C_1+C_2)}$ and $ \ol M>\max_{t\in I_1}x(t)$, there exists an interval $I_{j+1}\setminus \{t_j\}=]t_j,t_{j+1}]$ and $T^*\in ]t_j,t_{j+1}]$ such that $x(T^*)=\ol M =\max_{t^*\le s\le T^*}x(s)$. From \eqref{x},  we have
    $\ol M=x(T^*)\le M_{j-1} e^{2C_2}$, hence 
    $$M_{j-1}>Me^{2C_1}.$$
    From \eqref{m_j}, it then follows that $m_j>M$. Together with \eqref{M_j}, \eqref{x} and $\ol M\le M_{j+1}$,  we get $$m_{j+1}\ge M_je^{-2C_1}\ge M_{j+1}e^{-(2 C_1+C_2)}\ge \ol Me^{-(2 C_1+C_2)}>M,$$
   thus $x(t)>M$ for all $t\in [t_{j-1},T^*]$. From \eqref{general}, for $t=T^*$
   we get $x(T^*-\sigma_k(T^*))>M\, (1\le k\le m)$, hence
    \begin{equation*}
     \begin{split}
     x'(T^*)&< \sum_kp_k(T^*)x(T^*-\tau_k(T^*))e^{-a_k (T^*)M}-\de x(T^*)\\
     &\le \de(T^*)x(T^*)\left (\frac{\sum_kp_k(T^*)e^{-a_k(T^*)M}}{\de(T^*)}-1\right)<0,
     \end{split}
   \end{equation*}
 which contradicts the definition of $T^*$. Thus $x(t)$ is bounded.

 We now show the persistence of $x(t)$.
 Consider $m_j,M_j$ as in  the construction above. If $\liminf_{t\to\infty}x(t)=0$, we may consider $\ol m<me^{-2 C_1}$ such that $ \ol m<\min_{t\in I_1}x(t)$ and $S^*\in I_{j+1}\setminus \{t_j\}=]t_j,t_{j+1}]$ so that
 $$x(S^*)=\min_{t^*\le t\le S^*}x(t)=\ol m.$$
 From \eqref{m_j}, $M_j\le e^{2C_1}m_{j+1}\le e^{2C_1}\ol m<m$. On the other hand, for $t\in [t_j,S^*)$, $\ol m=x(S^*)\ge x(t)e^{-C_1}$, thus $x(t)\le \ol m e^{ C_1}<m$.
  In this way,
 we conclude that $x(t)< m$ on $[t_{j-1},S^*]$, which implies
 \begin{equation*}
     \begin{split}
     x'(S^*)&>\sum_kp_k(S^*)x(S^*-\tau_k(S^*))e^{-a_k(S^*) m}-\de x(S^*)\\
     &\ge \de (S^*)x(S^*)\left (\frac{\sum_kp_k(S^*)e^{-a_k(S^*)m}}{\de(S^*)}-1\right)>0.
     \end{split}
   \end{equation*}
   This is not possible, since $x'(S^*)\le 0$ by definition of $S^*$. Thus $x(t)$ is persistent as well.
   
   \smal
   
  Next, we suppose that  (A3) is satisfied and prove the  permanence of \eqref{general} -- with  uniform bounds as  in \eqref{boundsmM} -- in two further steps, by analysing separately solutions which converge at $\infty$ and solutions with an oscillatory behaviour.

   \smal

{\it Step 3}.
   If $x(t)$ is a positive solution of \eqref{general} such that
 there exists $L=\lim_{t\to \infty}x(t)$, then $L\in (0,\infty)$ from  Lemma \ref{lem2.1}. Fix $\vare\in (0,1)$, choose $t^*$ above such that
  $\al_0\ge \al-\vare >1, \ga_0\le \ga+\vare$,
and observe that, for some $ t_*\ge t^*$ sufficiently large and $t\ge t_*$, we have
$$x'(t)+\delta(t) x(t)\le \ga_0 \de(t)(L+\vare) e^{-a^-(L-\vare)}$$
 and 
  $$x'(t)+\delta(t) x(t)\ge \al_0 \de(t)(L-\vare) e^{-a^+(L+\vare)}.$$
  Integrating over $[t_*,t]$, we obtain
  $$x(t)\le x(t_*)e^{-\int_{t_*}^t\delta (s)\, ds}+\ga_0 (L+\vare)e^{-a^-(L-\vare)} [1-e^{-\int_{t_*}^t\delta (s)\, ds}]$$
  and 
  $$x(t)\ge x(t_*)e^{-\int_{t_*}^t\delta (s)\, ds}+\al_0 (L-\vare)e^{-a^+(L+\vare)} [1-e^{-\int_{t_*}^t\delta (s)\, ds}].$$ 
  By letting $t\to\infty,\vare \to 0^+ $, (A3) allows to conclude that
  \begin{equation}\label{Bounds_monotone}\frac{1}{a^+}\log \al\le L\le \frac{1}{a^-}\log \ga.
  \end{equation}
  In particular, $m\le L\le M$ for $m,M$ as in \eqref{boundsmM}.

   \med

   
 Step 4:    From Step 2, consider a positive solution $x(t)$ and define    
          \begin{equation}\label{lL}
          l=\liminf_{t\to\infty}x(t),\q L=\limsup_{t\to\infty}x(t),\end{equation}
 with $0<l\le L<\infty$. In virtue of Step 3, it suffices to consider the case  $l<L$. Hence, by the fluctuation lemma (see \cite{smith} there is a sequence $(t_n)$ of local maxima points with $x(t_n)\to L, x'(t_n)=0$ and a sequence of  local minima points $(s_n)$  with $x(s_n)\to l, x'(s_n)=0$.
 
 For any $\vare>0$, let $t^*$  be such that 
  \eqref{mM} and  \eqref{CC} hold. From the construction in Step 1, clearly $\limsup_jM_j=L, \liminf_j m_j=l$. 
 
Let $x(t_n)=M_{j_n}$ for some subsequence $(M_{j_n})_{n\in\N}$ of $(M_j)_{j\in\N}$ with $\lim_n M_{j_n}=L$.
 From Step 2, $\min \{m_j,m_{j-1}\}\ge M_je^{-2(C_1+C_2)}$.  Since $x'(t_n)=0$, from \eqref{general}, for $n$ large we obtain
\begin{equation}\label{Step4.1}
\begin{split}
0&\le-\de(t_n) x(t_n)+\sum_k p_k(t_n) x(t_n-\tau(t_n))e^{-a_k(t_n) M_{j_n}e^{-2(C_1+C_2)}}\\
&\le -\de(t_n) x(t_n)+(L+\vare)\sum_k p_k(t_n) e^{-a_k(t_n) M_{j_n}e^{-2(C_1+C_2)}},
\end{split}
\end{equation} 
thus $$0\le -x(t_n)+(L+\vare)\ga_1 e^{-a^- M_{j_n}e^{-2(C_1+C_2)}}.$$
 By letting $\vare\to 0, n\to\infty$, we derive $L \le L \ga e^{-a^- Le^{-2(D+P)}}$, or in other words, 
 $$L\le \frac{\log \ga}{a^-}e^{2(D+P)}.$$
 
 Similarly, let $x(s_n)=m_{j_n}$ for some subsequence $(m_{j_n})_{n\in\N}$ of $(m_j)_{j\in\N}$ with $\lim_n m_{j_n}=l $ and, from Step 3,  $\max_{t\in [t_{j_n-2}, s_n]}x(t)\le 
 \max\{M_j,M_{j-1}\}\le m_je^{2C_1+C_2}$. 
 Thus, for $n$ large, 
\begin{equation}\label{Step4.2}
\begin{split}
0&\ge-\de(s_n) x(s_n)+\sum_k p_k(s_n) x(s_n-\tau(s_n))e^{-a_k(s_n) m_{j_n}e^{2 C_1+C_2}}\\
&\ge \de(s_n)\left (- x(s_n)+ (l-\vare)\frac{p(s_n)}{\de(s_n)}e^{-a^+ m_{j_n}e^{2C_1+C_2}}\right),
\end{split}
\end{equation} 
and taking limits we conclude that $1  \ge \al e^{-a^+ le^{2D}}$, which means that
 $$l\ge  \frac{\log \al }{a^+}e^{-(2D+P)}.$$
     The proof is complete.
\end{proof}

\begin{rmk}\label{rmk2.1}  In the proof of Theorem \ref{thm2.1}, note that hypothesis (A3) is relevant only to deal with positive solutions for which there exists $\lim_{t\to \infty} x(t)$. On the other hand, assuming  (A0)--(A3), the above proof  shows that if  there exists $\lim_{t\to \infty} x(t)$ for all positive solutions, then the estimates \eqref{boundsmM} are valid with $m=\frac{1}{a^+}\log \al, M= \frac{1}{a^-}\log \ga$.
\end{rmk}

\begin{rmk}\label{rmk2.1'} In \cite[Propositions 2.2 and 2.3]{ElRuiz22}, El-Morshedy and Ruiz-Herrera established sharper uniform lower and upper bounds $m,M$ than the ones  in  \eqref{boundsmM}, but only for some  families of DDEs with a single pair of delays in a mixed monotone nonlinearity. Moreover, all the coefficients were supposed to be bounded and bounded from below away from zero.
\end{rmk}

For the particular case of \eqref{eq:nicholson}, the estimates in the above Step 4 can be refined and we obtain:

\begin{thm} \label{thm2.1'}  For \eqref{eq:nicholson}, let $p:=\sum_{j=1}^{m} p_j  >\delta$ and  assume
\begin{itemize}
\item[(H0)]$\de, p_j , a_j>0$,  $\be(t)>0$ is continuous  and $ \sigma_j(t), \tau_j(t)\ge 0$ are continuous and bounded, $t\ge t_0$;
\item[(H1)] $\int_0^\infty \be(s)\, ds=\infty$;
\item[(H2)]  $\limsup_{t\to\infty} \int_{t-\tau}^t\beta(s)\, ds=: C<\infty .$
   \end{itemize}
   Then, \eqref{eq:nicholson} is  permanent. Moreover,
   any positive solution of \eqref{eq:nicholson}  satisfies
      \begin{equation}\label{perm}
Ke^{-(2\de+p) C} \le \liminf_{t\to \infty}x(t)\le \limsup_{t\to \infty}x(t) \le Ke^{2(\de+p)C},
\end{equation}
where $K$ is the positive equilibrium of  \eqref{eq:nicholson}.
\end{thm}

\begin{proof}  Define \begin{equation}\label{f0}f(x)=\frac{1}{\de} (\sum_{j=1}^m p_j  e^{-a_j x}),\ x\ge 0,\end{equation}
and observe that $f$ is decreasing and $f(K)=1$. 

 
Fix  a positive solution $x(t)$ and define    $l,L$ as in \eqref{lL}.
 From Theorem \ref{thm2.1}, $0<l\le L<\infty$.  If $l=L$, we easily deduce that $L=K$ (see also Lemma \ref{lem2.2} and Remark \ref{rmk2.3}).
 Thus, we only need to consider solutions $x(t)$ with an oscillatory behaviour,  i.e.,  we suppose that $l<L$. 
 For any $\vare>0$, let $t^*$  be such that \eqref{CC} holds, where $D=\de C, P=pC$. Write $C_\vare=C+\vare$.
Consider the construction and notation in the proof of Theorem \ref{thm2.1}, so that $\limsup_jM_j=L, \liminf_j m_j=l$ and, 
as before, take sequences $(t_n)$ of local maxima points with $x(t_n)\to L, x'(t_n)=0$ and  of  local minima points $(s_n)$  with $x(s_n)\to l, x'(s_n)=0$.

    Let $x(t_n)=M_{j_n}$ for some subsequence $(M_{j_n})_{n\in\N}$ of $(M_j)_{j\in\N}$.
 Reasoning as in \eqref{Step4.1}, from  \eqref{eq:nicholson} and $0=x'(t_n)$, for $n$ large we obtain
\[
\begin{split}
0&\le-\de x(t_n)+\sum_k p_k x(t_n-\tau(t_n))e^{-a_k M_{j_n}e^{-2(\de+p)C_\vare}}\\
&\le -\de x(t_n)+(L+\vare)\sum_k p_k e^{-a_k M_{j_n}e^{-2(\de+p)C_\vare}}
\end{split}
\] 
 By letting $\vare\to 0, n\to\infty$, we derive $\de L \le L \sum_k p_k e^{-a_k Le^{-2(\de+p)C}}$, or in other words, $f(Le^{-2(\de+p)C})\ge 1$.
 Therefore
 $L\le Ke^{2(\de+p)C}.$
In a similar way, let $x(s_n)=m_{j_n}$ for some subsequence $(m_{j_n})_{n\in\N}$ of $(m_j)_{j\in\N}$. Arguing as in \eqref{Step4.2}, for $n$ large we derive
 \[
\begin{split}
0&\ge-\de x(s_n)+\sum_k p_k x(s_n-\tau(s_n))e^{-a_k m_{j_n}e^{(2\de+p) C_\vare}}\\
&\ge -\de x(s_n)+ (l-\vare)\sum_k p_ke^{-a_k m_{j_n}e^{(2\de+p) C_\vare}},
\end{split}
\] 
and taking limits we conclude that $1  \ge f(le^{(2\de+p) C})$, which means that
 $l\ge Ke^{-(2\de+p) C}.$ \end{proof}

%

\section{Preliminary  results}
\setcounter{equation}{0}
Once the permanence of \eqref{general} is established, the aim now is to give criteria for its global attractivity.   
In this section, some preliminary lemmas are proven. A first important observation is stated in the lemma below, whose proof is trivial and therefore omitted.

\begin{lem} \label{lem3.1} Assume (A0). If $x^*(t)$ is a fixed positive solution of
 \eqref{general}, then the  change of variables
\begin{equation}\label{y}y(t)=\frac{x(t)}{x^*(t)}\end{equation}
 transforms \eqref{general} into
 \begin{equation}\label{general2}
    y'(t)=\sum_{j=1}^{m} P_j(t)  y(t-\tau_j(t)) e^ {-A_j(t) y(t-\sigma_j(t))}-D(t) y(t),\qq t\ge 0,
\end{equation}
where
\begin{equation}\label{coefNew}
\begin{split}
 &P_j(t)=\frac{1}{x^*(t)}p_j(t)x^*(t-\tau_j(t)),\\
  &A_j(t)=a_j(t)x^*(t-\sigma_j(t)),\q j=1,\dots,m,\\
 &D(t)=\sum_j P_j(t)e^{-A_j(t)}.
 \end{split}
\end{equation}
\end{lem}

We observe that the change  \eqref{y}, already considered in \cite{FariaRost,Faria21b}, leads to  equation
\eqref{general2} which has the some form of \eqref{general} and possesses the positive equilibrium $K=1$. Henceforth, the main strategy  is first to establish sufficient conditions for the global attractivity of a positive equilibrium $K$ when such an equilibrium exists; then we will drop such an assumption, and proceed to derive a major result for global attractivity of positive solutions of \eqref{general}  by applying Lemma \ref{lem3.1}.

Consequently, at the present moment the following auxiliary temporary assumption is imposed: 
\begin{itemize}
\item[(K1)] There is a positive equilibrium $K$ of  \eqref{general}, i.e., there is $K>0$ such that
$$\de(t)=\sum_jp_j(t)e^{-a_j(t)K}\q {\rm for}\q t\ge 0.$$
\end{itemize}

\begin{lem} \label{lem2.2} Assume (A0), (A1), (A3) and (K1). If $x(t)$ is a positive solution of
 \eqref{general} with  $ \lim_{t\to \infty}x(t)=L$, then $L=K$.\end{lem}

\begin{proof} Suppose that $L=\lim_{t\to \infty}x(t)$. From  Lemma \ref{lem2.1}, $L\in (0,\infty)$.
In order to obtain a contradiction, suppose that $L>K$.

Then, there is $\vare>0$ such that
$x(t-\sigma_j(t))-K\ge \vare$  for $t\gg 1$. Let $\eta<0$ be given by $\eta=-1+e^{-a^-\vare}$.
Using (H4), for $t\gg 1$ we obtain
\begin{equation*}
\begin{split}
x'(t)&=\sum_jp_j(t)e^{-a_j(t)K}\left [-x(t)+x(t-\tau_j(t)) e^ {-a_j(t) [x(t-\sigma_j(t))-K]}\right]\\
&\le \sum_jp_j(t)e^{-a_j(t)K}\left [-x(t)+x(t-\tau_j(t)) e^ {-a^- \vare}\right]\\
&\le \sum_jp_j(t)e^{-a_j(t)K}L \frac{\eta}{2}=L \frac{\eta}{2}\de(t).
\end{split}
    \end{equation*}    
%
%
Hence,  for some $T\ge 0$ and $t\ge T$, (A3) yields
$$x(t)\le x(T)+L\frac{\eta}{2}\int_T^t\de(s)\, ds\to -\infty \q {\rm as}\q t\to\infty,$$
which is not possible. If $L<K$,  a similar contradiction is obtained.\end{proof}

\begin{rmk}\label{rmk2.3}In particular for  \eqref{eq:nicholson} with $p>\delta$ and $\int_{0}^\infty \beta(t)\, dt=\infty$, we conclude that for any positive solution $x(t)$ such that there exists  $ \lim_{t\to \infty}x(t)=L$, then $L=K$.
\end{rmk}

The next lemma plays a key role in the proof of our main result, presented in Section 4.

\begin{lem} \label{lem2.3} Assume (A0),  (K1) and 
\begin{equation}\label{h0}
\begin{split}
&\sup_{t\ge 0} \int_{t-\sigma_j(t)}^t p(s)\, ds<\infty, \ j=1,\dots,m.
 \end{split}
    \end{equation}
    For   a positive solution $x(t)$ of
 \eqref{general},  let
 $ l= \liminf_{t\to \infty}x(t),\ L= \limsup_{t\to \infty}x(t)$
as in \eqref{lL}. If  $0<l<L<\infty$, then $l<K<L$.
\end{lem}

\begin{proof} Let  $0<l<L<\infty$. We proceed along the major lines of the proofs  of \cite[Theorems 3.1 and 3.3]{FariaPrates}, however additional arguments are required. 
Define \begin{equation}\label{f}f(t,x)=\frac{1}{\de(t)} (\sum_{j=1}^m p_j (t) e^{-a_j(t) x}),\ t,x\ge 0.\end{equation}
From (K1),  $f(t,K)=1$ for all $t\ge 0$. 

Consider sequences $(t_n),(s_n)$ such that $t_n, s_n \to \infty $ and $x'(t_n)= x'(s_n)=0$, $ x(t_n)\to L, x(s_n)\to l.$ Taking  subsequences if needed, one may assume that, for $j=1,\dots,m$,
\begin{equation}\label{limitsj_h,g}
\begin{split}
    &x(t_n -\tau_j(t_n ))\to L_{\tau_j}, \  x(t_n -\sigma_j(t_n ))\to L_{\sigma_j} , \ \\
     &x(s_n -\tau_j(s_n ))\to l_{\tau_j}, \ x(s_n -\sigma_j(s_n ))\to l_{\sigma_j},\\
     &a_j(t_n)\to a_{j}^*,\ a_j(s_n)\to a_{j}^{**}
    \end{split}
\end{equation}
where $l_{\sigma_j},  l_{\tau_j},L_{\sigma_j},  L_{\tau_j}\in [l,L]$ and $0<a_j^-\le a_{j}^*,a_{j}^{**}\le a_j^+$ for all $j$.
From \eqref{general},
 \begin{equation*}
    \begin{split}&x'(t_n)=0=-x(t_n)+\sum_j\frac{p_j(t_n)}{\de(t_n)}e^{-a_j(t_n)K}x(t_n-\tau_j(t_n))e^{-a_j(t_n)[x(t_n-\sigma_j(t_n))-K]},\\
   &x'(s_n)=0=-x(s_n)+\sum_j\frac{p_j(s_n)}{\de(s_n)}e^{-a_j(s_n)K}x(s_n-\tau_j(s_n))e^{-a_j(s_n)[x(s_n-\sigma_j(s_n))-K]}.
   \end{split}
 \end{equation*}
Define 
$$L_\sigma^-:=\min_j L_{\sigma_j},\q l_\sigma^+:=\max_j l_{\sigma_j}.$$
Assumption (K1) implies that, for  $\vare>0$ small and  $n$ large, 
\begin{equation*}
    \begin{split}&0\le -x(t_n)+(L+\vare) e^{-\min_j \big[a_j(t_n)\big(x(t_n-\sigma_j(t_n))-K\big)\big]}\\
    &0\ge -x(s_n)+(l-\vare) e^{-\max_j \big[a_j(s_n)\big(x(s_n-\sigma_j(s_n))-K\big)\big]} \, .
      \end{split}
 \end{equation*} 
By taking limits  one obtains
$L\le Le^{-\min_j[a_j^*(L_{\sigma_j}-K)]}$
and $
l\ge le^{-\max_j[a_j^{**}(l_{\sigma_j}-K)]}$, thus
$$\min_j[a_j^*(L_{\sigma_j}-K)]\le 0\le \max_j[a_j^{**}(l_{\sigma_j}-K)].$$
%
%
Consequently,
  \begin{equation}\label{Ll_sigma}L_\sigma^-\le K\le  l_\sigma^+.\end{equation}
In particular $l\le K\le L$.


Denote $ \zeta_j^+=\limsup_{t\to\infty} \int_{t-\sigma_j(t)}^t p(s)\, ds$, with $\zeta_j^+\in [0,\infty)$ from \eqref{h0},  for $1\le j\le m$.
Fix $\vare>0$ and $i\in\{1,\dots,m\}$. For $n$ sufficiently large such that $x(t-\tau_j(t))\le L+\vare$ for $t\ge t_n-\tau$, integrating \eqref{general} over  $[t_n-\sigma_i(t_n),t_n]$ leads to
\begin{align*}
    x(t_n)&= x(t_n-\sigma_i(t_n)) e^{-\int_{t_n-\sigma_i(t_n)}^{t_n} \delta (v)\, dv}\\&+\sum_j \int_{t_n-\sigma_i(t_n)}^{t_n}  p_j (s) x(s-\tau_j(s))e^{-a_j(s)x(s-\sigma_j(s))}e^{\int_{t_n}^s \delta (v)\, dv} \, ds \\
    &\le x(t_n-\sigma_i(t_n)) e^{-\int_{t_n-\sigma_i(t_n)}^{t_n} \delta (v)\, dv}\\
    &+(L+\vare)\sum_j \int_{t_n-\sigma_i(t_n)}^{t_n}  p_j (s) e^{-a_j(s)K}e^{-a_j(s)(x(s-\sigma_j(s))-K)}e^{\int_{t_n}^s \delta (v)\, dv} \, ds.
    \end{align*}
By applying the mean value theorem for integrals, we obtain
\begin{align*}
x(t_n)&= x(t_n-\sigma_i(t_n)) e^{-\int_{t_n-\sigma_i(t_n)}^{t_n} \delta(v)\, dv}+(L+\vare)\sum_j  e^{- a_{ij}^n(\tilde {l}_{ij}^n-K)} \int_{t_n-\sigma_i(t_n)}^{t_n}   p_j(s)e^{-a_j(s)K} e^{\int_{t_n}^s \delta (v)\, dv} \, ds\\
&\le x(t_n-\sigma_i(t_n)) e^{-\int_{t_n-\sigma_i(t_n)}^{t_n} \delta (v)\, dv}+(L+\vare)e^{-\min_{i,j}a_{ij}^n(\tilde {l}_{ij}^n-K)} \int_{t_n-\sigma_i(t_n)}^{t_n}\de(s) e^{\int_{t_n}^s \delta (v)\, dv} \, ds
\\ &= x(t_n-\sigma_i(t_n)) e^{-\int_{t_n-\sigma_i(t_n)}^{t_n} \delta (v)\, dv}+ (L+\vare)e^{-\min_{i,j} a_{ij}^n(\tilde {l}_{ij}^n-K)}\Big[1- e^{-\int^{t_n}_{t_n-\sigma_i(t_n)} \de(v)\, dv}\Big ],
\end{align*}
where $a_{ij}^n=a_j(r_{ij}^n)$, $\tilde {l}_{ij}^n=x(r_{ij}^n-\sigma_j(r_{ij}^n))$, for some $r_{ij}^n\in [t_n-\sigma_i(t_n),t_n]$. Without loss of generality, let $a_{ij}^n\to \tilde a_{ij}, \tilde {l}_{ij}^n\to 
\tilde {l}_{ij}$  and 
$\lim\int^{t_n}_{t_n-\sigma_i(t_n)}\de(v)\, dv =\zeta_i^*$
with $\zeta_i^*\in [0,\zeta_i^+]$. Define $\tilde l=\min_{i,j}\tilde {l}_{ij}\in [l,L]$, $\tilde A^-=\min_{i,j} \tilde a_{ij}(\tilde {l}_{ij}-K)$. 
Taking limits $n\to \infty, \vare\to 0$ in the formula above, one obtains
\begin{equation}\label{L}L\le L_{\sigma_i}e^{- \zeta_i^*}+Le^{-\tilde A^-}(1-e^{- \zeta_i^*}),\q i=1,\dots,m.\end{equation}



 In a similar way, we derive that there exist  $\tilde{\tilde a}_{ij}\in [a^-,a^+]$, $\tilde {L}_{ij}\in [l,L]$ and $\lim\int^{s_n}_{s_n-\sigma_i(s_n)}\de(v)\, dv =\zeta_i^{**} \in  [0,\zeta_i^+]$, such that for $\tilde A^+:=\max_{i,j} \tilde{\tilde a}_{ij}(\tilde {L}_{ij}-K)$
the inequalities
\begin{equation}\label{l}l\ge l_{\sigma_i}e^{- \zeta_i^{**}}+le^{-\tilde A^+}(1-e^{- \zeta_i^{**}}),\q i=1,\dots,m\end{equation}
hold. Moreover, define $\tilde L:=\max_{i,j}\tilde {L}_{ij}$.

\med

{\it Claim 1}. We first claim that
\begin{equation}\label{claim01}\tilde l\le K\le \tilde L.\end{equation}

The proof of  Claim 1 follows in several cases.

\med

Case 1: Assume that $\zeta_{i}^*= 0$ and $\zeta_{i}^{**}= 0$ for all $i=1,\dots,m$.

From \eqref{Ll_sigma},  \eqref{L}  and \eqref{l}, we obtain 
$L_\sigma^-=L=K$ and $l_\sigma^+=l=K$, which contradicts the hypothesis $l<L$.

Case 2: Suppose there are $i_1,i_2\in \{1,\dots,n\}$ such that $\zeta_{i_1}^*> 0$ and $\zeta_{i_2}^{**}> 0$.  


If $\tilde l>K$, then $\tilde A^->0$ and    \eqref{L} with $i=i_1$ implies that
$L<L_{\sigma_{i_1}} e^{-\de \zeta_{i_1}^*}+L(1-e^{-\de \zeta_{i_1}^*})\le L$, which is not possible. Similarly, if $\tilde L<K$, then
$\tilde A^+<0$ and from  \eqref{l} with $i=i_2$ one derives $l>l$ , again a contradiction.

Case 3: Let $\zeta_{i}^*= 0$  for all $i=1,\dots,m$ and assume there is $i_2\in \{1,\dots,n\}$ such that $\zeta_{i_2}^{**}> 0$. (The case $\zeta_{i}^{**}= 0$  for all $i=1,\dots,m$ and  $\zeta_{i_1}^{*}> 0$ for some  $i_1\in \{1,\dots,n\}$ is analogous.)

As in the above cases, one obtains $L_{\sigma_i}=L_\sigma^-=L=K$  from the inequalities  \eqref{L},
and  $ \tilde A^+\le 0$ from the definition of  $ \tilde A^+$. If $ \tilde A^+< 0$, from  \eqref{l} with $i=i_2$ one obtains a contradiction, and therefore $ \tilde A^+=0$.
%
%
Thus, from \eqref{Ll_sigma}, there is $i_0$ such that $l_{\sigma_{i_0}}=l_\sigma^+=K=L>l$.
Now, using    \eqref{l} with $i=i_0$ leads to
$$l\ge L e^{- \zeta_{i_0}^{**}}+l(1-e^{- \zeta_{i_0}^{**}})> l,$$
a contradiction. 
This concludes the proof of  claim \eqref{claim01}.

%
%
%
%
%

\med


Next, from \eqref{claim01}, we have
$$\tilde A^-\ge -a^+(K-\tilde {l}) ,\q \tilde A^+\le a^+(\tilde L-K).$$
Together with  \eqref{L}  and   \eqref{l}, this leads to the estimates
 \begin{equation}\label{L:first}
L\le L_{\sigma_i}e^{-\zeta_i^*}+Le^{a^+(K-\tilde {l})}(1-e^{- \zeta_i^*}),\q i=1,\dots,m,
\end{equation}
and 
\begin{equation}\label{l:first}l\ge l_{\sigma_i}e^{- \zeta_i^{**}}+le^{a^+(K-\tilde {L})}(1-e^{- \zeta_i^{**}}),\q i=1,\dots,m,
\end{equation}
with $l\le \tilde l\le K\le \tilde L\le L$.

{\it Claim 2}. We finally show that \begin{equation}\label{claim2}l<K<L.\end{equation}
Suppose that $l=K<L$ (the situation $l<K=L$ is similar). From \eqref{claim01}, this yields $\tilde l=K$ and, from \eqref{L:first},  $L_{\sigma_{i}}=L$ for all $i$. Using again  \eqref{Ll_sigma}, one concludes that $L=L_{\sigma}^-\le K$
which is not possible.  The proof of \eqref{claim2}  is complete.\end{proof}


\begin{rmk}\label{rmk2.2} {\rm If $0<l<L<\infty$, for $i_1,i_2$ such that   $L_{\sigma}^-=L_{\sigma_{i_1}},l_{\sigma}^+=l_{\sigma_{i_2}}$, formulas  \eqref{Ll_sigma},  \eqref{L:first}, \eqref{l:first}, \eqref{claim2} also show that $\zeta_{i_1}^*> 0$ and $\zeta_{i_2}^{**}> 0$.}
\end{rmk}



As in \cite{ER19}, to proclude the existence of  solutions with an oscillatory behaviour about $K$,
  we shall use global attractivity results for difference equations of the form
\begin{equation}\label{diffeq}
x_{n+1}=h(x_n),\qq n\in \N_0,
\end{equation} 
where $h:I\to I$ is a continuous function on a   (possibly unbounded) interval $I\subset \R$. 
A fixed point $x^*$ of $h$ is said to be a {\it global attractor} for \eqref{diffeq} if all solutions $(x_n)$ of   \eqref{diffeq} with   initial data $x_0\in I$ satisfy $x_n\to x^*$. 
Several authors have established criteria for a fixed point of \eqref{diffeq} to be, or not to be, a global attractor, see e.g. \cite{ML,M-PN1,singer} and references therein. For completeness, see also the selected assertions from \cite{ML} collected in \cite[Lemma 3.1]{FariaPrates}.
In this context, the use of Schwarzian derivatives has proven to be a powerful tool to study the attractivity of solutions to \eqref{diffeq} \cite{LizTT,singer}. If $h$ is a $C^3$  function on an interval $I$ whose derivative does not vanish,
the Schwarzian derivative of $h$ is defined by 
$$Sh(x):=\frac{h'''(x)}{h'(x)}-\frac{3}{2}\left(\frac{h''(x)}{h'(x)}\right)^2.$$

The next auxiliary result will be applied to a difference equation of the form \eqref{diffeq} used in the proof of our main theorem.
\begin{lem} \label{lem2.4} Let $K>0$ and define \begin{equation}\label{g}g(x):=e^{a^+(K-x)},\q x>0.\end{equation}
For a fixed constant $\th_0\in (0,1)$, assume that \begin{equation}\label{H3:theta}
a^+K(\th_0^{-1} -1)\le 1.\end{equation}
Then $(1-\th_0)g(\th_0K)< 1$ and the function
\begin{equation}\label{h}
h(x):=\frac{\th_0 K}{1-g(x)(1-\th_0)}, \q x\ge \th_0K,
\end{equation}
is defined and decreasing on $I:=[\th_0K,\infty)$, satisfies $h(I)\subset I$, $h(K)=K$, $|h'(K)|\le 1$ and $Sh(x)<0$ on $I$.
\end{lem}

\begin{proof} For $\th_0$ as in the statement, from \eqref{H3:theta} we obtain
$$(1-\th_0)g(\th_0K)= (1-\th_0)e^{a^+(K-\th_0K)}\le (1-\th_0)e^{\th_0}<1,$$
because $\log(1-x)+x<0$ for all $x\in (0,1)$.
Hence, since $g$ is decreasing, $g(x) (1-\th_0)<1$ for $x\ge \th_0K$ and $h$ is well-defined, decreasing  and positive on $I=[\th_0K,\infty)$. Moreover, $h(\infty)=\th_0K$, therefore $h(I)\subset I$. Clearly $g(K)=1$, $K$ is the unique fixed point of $h$ and $h'(K)=-a^+K(\th_0^{-1}-1)\in [-1,0)$. Finally, since $h(x)=h_1(g(x))$ with $h_1$ a fractional linear transformation, we have \cite{singer} $Sh(x)=Sg(x)=-\frac{(a^+)^2}{2}<0$.\end{proof}

%
%
%

\section
{Global  attractivity criteria for \eqref{general}}
\setcounter{equation}{0}

 In this section, we establish our main results on the global attractivity of solutions for  the general nonautonomous model \eqref{general}.  An additional hypothesis on the size of the delays $\sigma_i(t)$ will be imposed, and
 two major criteria  given: the first one depends on  the lower and upper bounds of a particular fixed solution $x^*(t)$, while the second statement does not require any a priori knowledge of  such bounds, and instead uses the uniform bounds established in  Theorem \ref{thm2.1}. 
 
%
%
%
%
%
%
The major result of this paper is stated below.

 \begin{thm}\label{thm3.3}  Assume  (A0)--(A3). \vskip 1mm
 
(a)   Let $x^*(t)$ be a positive solution  of \eqref{general} satisfying the following condition:
 \begin{itemize}
\item[(A4)] There are  $m^*,M^*>0$ and $T>0$ such that $m^*\le x^*(t)\le M^*$ for $ t\ge T-\tau,$ and  \begin{equation}\label{A4}
a^+M^*\Big (\frac{M^*}{m^*}e^{\zeta^+} -1\Big )\le 1,
\end {equation}
where 
\begin{equation}\label{zeta_M}
\zeta^+ =\max_{1\le j\le m} \limsup_{t \to \infty} \int_{t-\sigma_j(t)}^{t}  \de(s)\, ds;
\end{equation}
\end{itemize}
Then $x^*(t)$   is globally attractive.\vskip 1mm

(b) Suppose that:
 \begin{itemize}
\item[(A5)] For $\al,\ga, D,P$ are as in \eqref{alphagamma}  and $\zeta^+$  as in \eqref{zeta_M}, it holds
 \begin{equation}
\frac{a^+}{a^-}\log \ga \, e^{2(D+P)}\Big [\frac{a^+\log \ga}{a^-\log \al}e^{4D+3P+\zeta^+} -1\Big ]< 1.
\end {equation}
\end{itemize}
 Then \eqref{general} is globally attractive; that is, 
$\lim_{t\to \infty} (x_1(t)-x_2(t))=0$
 for any two positive solutions $x_1(t), x_2(t)$.
\end{thm}
 
%
%

 The proof of this theorem will follow after a preparatory step, where, in view of Lemma \ref{lem3.1},
 we first assume the provisional requirement (K1), together with  condition (A4) for the case $m^*=M^*=K$.
%





\begin{prop}\label{prop4.1}  Assume  (A0)--(A3), (K1) and
\begin{itemize}\item[(K2)] $a^+K(e^{\zeta^+} -1)\le 1,$
where  $\zeta^+$  is as in \eqref{zeta_M}.
\end{itemize}
Then the equilibrium $K$   of \eqref{general} is globally attractive.
\end{prop}

\begin{proof} 
Let $x(t)$ be a solution with initial condition in $C_0^+$. By the permanence in Theorem \ref{thm2.1}, for 
$l,L$ as in \eqref{lL}
we have $0<l\le L<\infty$. Clearly our assumptions imply \eqref{h0}.
In virtue of Lemmas \ref{lem2.2} and \ref{lem2.3}, it suffices to show that the situation $l<K<L$ is not possible.

%
%
%
%
%

For that, we resume the proof of Lemma \ref{lem2.3}.  We have concluded that $L_{\sigma^-}\le K\le l_{\sigma^+}$
 and that there are  constants $\zeta_{i}^*,\zeta_{i}^{**}\in [0,\zeta^+]$ and $\tilde l\in[l,K],\tilde L\in[K,L]$  such that formulas \eqref{L:first}, \eqref{l:first},  are satisfied.
  Moreover, with $i_1,i_2$ such that   $L_{\sigma}^-=L_{\sigma_{i_1}},l_{\sigma}^+=l_{\sigma_{i_2}}$,
  it was observed in Remark \ref{rmk2.2} that $\zeta_{i_1}^*,\zeta_{i_2}^{**}\in (0,\zeta^+]$.


Let $K$ be as in (K1) and  $g(x)=e^{a^+(K-x)}$  as in \eqref{g}. From \eqref{L:first} with $i=i_1$, \eqref{l:first} with $i=i_2$, we deduce that
\begin{subequations} \label{l,L:second}
    \begin{equation}
       L\le Ke^{-\de \zeta_{i_1}^*}+Lg(\tilde l)(1-e^{- \zeta_{i_1}^*})
    \end{equation}
    \begin{equation}
       l\ge Ke^{-\de \zeta_{i_2}^{**}}+lg(\tilde L)(1-e^{- \zeta_{i_2}^{**}}).
    \end{equation}
\end{subequations}
 In particular, $l>Ke^{- \zeta^+}$.
In Lemma \ref{lem2.4}, choose $\th_0=\th_0(\zeta):=e^{- \zeta}$, and 
note that (K2) guarantees that \eqref{H3:theta} is satisfied for $\zeta\in (0,\zeta^+]$. 
Thus, the function $$h(x,\zeta):=\frac{ Ke^{-\zeta}}{1-g(x)(1-e^{- \zeta})}$$
is well defined for $x\ge Ke^{- \zeta^+}$ and $\zeta \in (0,\zeta^+]$. 
The estimates in \eqref{l,L:second} now lead to
    \begin{equation}\label{l,L:third}
       L\le h(\tilde l, \zeta_{i_1}^*),\q 
       l\ge h(\tilde L, \zeta_{i_2}^{**}).
    \end{equation}
Since $\frac{\p h}{\p x} (x, \zeta)  <0,
(x-K)   \frac{\p h}{\p\zeta} (x, \zeta)  <0$  for all $ x\ne K$ and $\zeta \in (0,\zeta^+]$, from \eqref{claim01} and \eqref{l,L:third} we obtain
$L\le h(\tilde l, \zeta^+)\le h(l, \zeta^+)$, $ l\ge h(\tilde L, \zeta^+)\ge h(L, \zeta^+)$
and finally 
\begin{equation} \label{l,L:4th}L\le \mathfrak{h}(l), \q l\ge \mathfrak{h}(L),\end{equation}
where $\mathfrak{h}:[Ke^{- \zeta^+},\infty)\to (Ke^{-\zeta^+},\infty)$ is given by
 \begin{equation}\label{frakh}
 \mathfrak{h}(x):=h(x,\zeta^+)=\frac{ Ke^{-\zeta^+}}{1-g(x)(1-e^{- \zeta^+})}. \end{equation}
Clearly, $\mathfrak{h}$ is strictly decreasing and $K$ is its unique fixed point. From \eqref{l,L:4th}, $\mathfrak{h}([l,L])\supset [l,L]$ with $l<K<L$.  Consider the difference equation 
 \begin{equation}\label{DE_h}
 x_{n+1}=\mathfrak{h}(x_n).
  \end{equation}
From \cite[Lemma 2.5]{ML}, we therefore conclude that $K$ is not a global attractor of \eqref{DE_h}.


   In contrast, as shown in Lemma \ref{lem2.4}, $S\mathfrak{h}(x)<0$ for $x>0$,
 $|\mathfrak{h}'(K)|\le 1$, and under these conditions  \cite[Corollary 2.9]{ML} implies that $K$ is a global attractor of \eqref{DE_h},  contradicting the above statement.
 Therefore, the situation $l<L$ is not possible.
 This ends the proof.
\end{proof}

\begin{proof}[Proof of Theorem \ref{thm3.3}] Fix any positive solution $x^*(t)$ of \eqref{general} as in (a). 
  After the change of variables \eqref{y},
consider the transformed equation \eqref{general2} whose  coefficients are defined in \eqref{coefNew}.
Define $P(t):=\sum_j P_j(t)$. Clearly  \eqref{general2} has the form  \eqref{general} and
$P(t)\ge e^{a^-m^*}D(t),  P(t)\le e^{a^+M^*}D(t)$ for $t\ge T$, hence \eqref{general2} satisfies (A1).
Moreover,  (K1) holds for \eqref{general2} with  $K=1$. 
Observe that
$$
D(t)=\frac{1}{x^*(t)}\Big [ (x^*)'(t)+\de(t) x^*(t)\Big ],
$$
thus 
\begin{equation}\label{int_D(t)} \begin{split}
&\int_T^\infty  D(s)\, ds
\ge \log \frac{m^*}{M^*}+ \int_T^\infty  \de(s)\, ds,\\
& \int_{t-\sigma_j(t)}^{t}  D(s)\, ds
=\log \Big (\frac{x^*(t)}{x^*(t-\sigma_j(t))}\Big )+ \int_{t-\sigma_j(t)}^{t}  \de(s)\, ds.
\end{split}\end{equation}
Consequently, if (A2), (A3) are valid for \eqref{general}, then (A2), (A3) are valid for \eqref{general2}
as well. 
To apply Proposition \ref{prop4.1}, we just need to check that  \eqref{general2} satisfies (K2), which translates here as $A^+(e^{Z^+} -1)\le 1$, where  according to the previous notation and from \eqref{coefNew} and \eqref{int_D(t)},
$$A^+:=\max_j\sup_{t\ge T} A_j(t)\le a^+M^*$$ and
\begin{equation}\label{Z+}
\begin{split}
Z^+:
&= \max_{1\le j\le m} \limsup_{t \to \infty} \left[\log \Big (\frac{x^*(t)}{x^*(t-\sigma_j(t))}\Big )+ \int_{t-\sigma_j(t)}^{t}  \de(s)\, ds\right]\\
&\le \log\Big (\frac{M^*}{m^*}\Big )+\zeta^+.
\end{split}\end{equation}
From (A4), this finishes the proof of (a).

To prove (b), we invoke Theorem \ref{thm2.1} to obtain  the  uniform lower and upper bounds $m,M$ in \eqref{permGen} given by the estimates \eqref{boundsmM}. Simple computations show that
(A5) implies $a^+M\Big (\frac{M}{m}e^{\zeta^+} -1\Big )< 1.$ Fix  any positive solution $x^*(t)$, and let $\vare>0$  be arbitrarily small and $T>0$ large such that
$a^+M_\vare\Big (\frac{M_\vare}{m_\vare}e^{\zeta^+} -1\Big )\le 1$
and
$0<m_\vare \le x^*(t)\le M_\vare$ for $t\ge T-\tau$, where $m_\vare=m-\vare, M_{\vare}=M+\vare $.  Replacing  $m^*,M^*$ in (a) by $m_\vare,M_\vare$,  the result follows.
\end{proof}

 When (K1)  holds and there are multiple pairs of delays, in general one cannot  explicitly evaluate $K$, thus  a criterion with (K2)  not depending on $K$ is useful.

\begin{cor}\label{cor3.2}  Assume  (A0)--(A3), (K1) and
\begin{itemize}\item[(K2*)] $ \frac{a^+}{a^-}(e^{\zeta^+} -1)\log \ga \le 1,$
where $\zeta^+$ is as in \eqref{zeta_M}.
\end{itemize}
Then the equilibrium $K$   of \eqref{general} is globally attractive.
\end{cor}

\begin{proof} From (K1), we derive $1\le \ga e^{-a^-K}$, thus ${a^-K}\le \log \ga$. In this way,
$a^+K(e^{\zeta^+} -1)\le   \frac{a^+}{a^-}(e^{\delta\zeta^+} -1)\log \ga$, thus (K2*) 
 implies (K2).
 \end{proof}

The result below is an immediate consequence of Proposition \ref{prop4.1}.

\begin{cor}\label{cor3.1} Consider \eqref{eq:nicholson} with $p>\de$  and under   hypotheses (H0)--(H2).
 In addition, assume
\begin{equation}
a^+K(e^{\delta\zeta^+} -1)\le 1
\tag{H3}
\end {equation}
where 
\begin{equation}\label{zeta_M_beta}
\zeta^+ =\max_{1\le j\le m} \limsup_{t \to \infty} \int_{t-\sigma_j(t)}^{t}  \beta(s)\, ds.
\end{equation}
Then, the positive equilibrium $K$   of \eqref{eq:nicholson} (defined by \eqref{equilK}) is globally attractive. In particular, this is the case if 
 \begin{equation*}\label{H3*}
 \frac{a^+}{a^-}(e^{\delta\zeta^+} -1)\log \frac{p}{\de} \le 1.
 \end{equation*}
\end{cor}

%
%
%

\begin{rmk}\label{rmk3.1}  {\rm In \cite{FariaPrates}, the global attractivity of $K$ as an equilibrium of \eqref{eq:nicholson}  was established assuming (H0) with $0<\be^-\le \be(t)\le \be^+<\infty$, (H3) and 
$$\frac{a^+}{a^-}<\frac{3}{2}.$$ This latter constraint was imposed in \cite{FariaPrates} to assure that $Sf(x)<0$ for the function $f$ defined  in \eqref{f0} and is now completely removed from Corollary \ref{cor3.1}. On the other end, 
clearly $0<\be^-\le \be(t)\le \be^+<\infty$  implies that (H1),~(H2) hold. Thus, even for the particular case \eqref{eq:nicholson}, Corollary \ref{cor3.1} is significantly sharper than  the result in  \cite{FariaPrates}.
 We also stress that El-Morshedy and Ruiz-Herrera's major result  for the autonomous equation  \eqref{Nich_ER} in \cite{ER19} is just a very particular case of Corollary \ref{cor3.1}, under the additional constraint $\tau\ge \sigma$.
 }\end{rmk}


%
%
%
%
%
%

 \begin{rmk}\label{rmk3.2} 
 For  a recent result on the  global attractivity of a  Nicholson equation with a pair of constant discrete mixed delays given by $x'(t)=-\de(t)x(t)+p(t)x(t-\tau)e^{-ax(t-\sigma)}\ (a>0,0\le\sigma\le \tau)$, see   \cite[Theorem 4.1]{ElRuiz22}; for this particular case, our results are not exactly comparable, since $\de(t),\be(t)$  in \cite{ElRuiz22} were supposed to 
 be bounded below and above by positive constants.
 \end{rmk}


By using a criterion for the existence of (at least) one positive periodic solution established  in \cite{Faria22}, an application of our main result to periodic equations is given below. 


\begin{thm}\label{thm4.3} Consider \eqref{general} under the general conditions (A0) and $\de(t),p_j(t),a_j(t),\tau_j(t),\sigma_j(t)$ ($ j=1,\dots,m$)   $\omega$-periodic functions for some $\omega >0$. Suppose that $p(t):=\sum_{=1}^m p_j(t)>\de(t)$.  For $\zeta^+$ as in \eqref{zeta_M} and
$ \ga=\max_{t\in [0,\omega]}\frac{p(t)}{\de(t)}, 
D=\int_0^\omega \de(t)\, dt,$
assume also that
\begin{equation}\label{periodic1}p(t)\ge e^D\de(t)\q {\rm with}\q
  p(t)\not\equiv e^D\de(t),\q t\in [0,\omega],
  \end{equation}  and \begin{equation}\label{periodic2}\frac{a^+}{a^-}e^D\log \ga (e^{D+\zeta^+}-1)\le 1.\end{equation} 
Then, there exists a positive $\omega$-periodic solution of \eqref{general} which is a global attractor of all positive solutions. 
\end{thm}

\begin{proof} The assumptions on the coefficients and delays imply that (A1)--(A3) are satisfied. Let $C_\omega^+(\R)$ be the space of continuous and $\omega$-periodic functions $y:\R\to [0,\infty)$. From \cite[Theorem 4.3]{Faria22} (see also \cite{FO19}),  condition \eqref{periodic1} guarantees that equation \eqref{general}
  possesses at least one $\omega$-periodic positive solution $x^*(t)$ in the cone $\mathcal{K}=\{y\in C_\omega^+(\R): y(t)\ge e^{-D}\|y\|_\infty, t\in \R\}.$ 
  
   From the definition of $\mathcal{K}$, it follows that $m^*\le x^*(t)\le M^*$ for all $t$, with $M^*/m^*\le e^D$. Moreover, an $\omega$-period solution $x^*(t)$ satisfies
\begin{equation*}\begin{split}
  x^*(t)&= (e^D-1)^{-1}\int_t^{t+\omega} e^{\int_t^s \de(u)\, du} \sum_{j=1}^mp_j(s)x^*(s-\tau_j(s))e^{-a_j(s) x^*(s-\sigma_j(s))} ds\\
  &\le (e^D-1)^{-1}M^* e^{-a^-m^*}\ga \int_t^{t+\omega} \de(s)e^{\int_t^s \de(u)\, du}\, ds\\
  &= (e^D-1)^{-1}M^* e^{-a^-m^*}\ga (e^D-1)=M^* e^{-a^-m^*}\ga,
  \end{split}
  \end{equation*}
  thus $e^{a^-m^*}\le \ga$ and $M^*\le \frac{1}{a^-}e^{D}\log \ga$.   In this way,  \eqref{periodic2} implies that (A4) is satisfied. The result follows then  from Theorem \ref{thm3.3}.(a).
\end{proof}

\begin{rmk}\label{rmk4.3}
In fact, from \cite{Faria22},  an $\omega$-periodic equation \eqref{general}
  has an $\omega$-periodic positive solution  in the cone $\mathcal{K}$  defined above if either \eqref{periodic1} or $ \int_0^\omega p(t)\, dt\ge e^D(e^D-1)$ holds. \end{rmk}
  

\med
  
  As  illustration of our results, several  examples are presented.
  
\begin{exmp}\label{ex1} {\rm  Consider an equation  \eqref{general} with  all the delays $\sigma_j(t)$ equal to zero:
 \begin{equation}\label{general_no_sigma}
    x'(t)=\sum_{j=1}^{m} p_j(t)  x(t-\tau_j(t)) e^ {-a_j(t) x(t)}-\delta(t) x(t),\qq t\ge 0,
\end{equation}
and assume (A0)--(A3).
For any fixed positive solution $x^*(t)$, we have
$\log \Big (\frac{x^*(t)}{x^*(t-\sigma_j(t))}\Big )=0$, thus $Z^+=\zeta^+=0$ for $\zeta^+$ in \eqref{zeta_M} and  $Z^+$ in  \eqref{Z+}. After the change of variables \eqref{y}, the transformed equation becomes
 \begin{equation}\label{general2_no_sigma}
    y'(t)=\sum_{j=1}^{m} p_j(t)  \left (y(t-\tau_j(t)) e^ {-a_j(t) x^*(t)y(t)}-e^ {-a_j(t) x^*(t)}y(t)\right),\qq t\ge 0.
\end{equation}
Eq. \eqref{general2_no_sigma} always satisfies (K1) and (K2)   with $K=1$. Consequently, Proposition \ref{prop4.1} leads to:

\begin{cor}\label{cor3.3}  For \eqref{general_no_sigma} under (A0)--(A3), any positive solution is a global attractor.\end{cor}}\end{exmp}

\begin{exmp}\label{ex2} {\rm Consider \eqref{eq:nicholson} with $m=2$,
$p_1=\frac{3}{5},p_2=\frac{1}{2}, a_1=\frac{1}{N}\, (N\in\N),a_2=1, \de =1$,   and $\be(t),\tau_j(t),\sigma_j(t)\, (j=1,2)$ satisfying the general conditions  set in (H0)-(H2). With the above notations,
 $p=1.1>\de$. Define $f(x)=\frac{3}{5}e^{-\frac{x}{N}}+\frac{1}{2}e^{-x}$, so that the positive equilibrium $K$ is the solution  of $\frac{3}{5}e^{-\frac{K}{N}}+\frac{1}{2}e^{-K}=1$. Observe that 
 $f(\frac{1}{4})<1$ if and only if $$e^{-\frac{1}{4N}}<\frac{5}{3}(1-\frac{1}{2}e^{-\frac{1}{4}})\approx 1.0177,$$
 which is true for all values of $N$. Thus, the positive equilibrium $K=K_N$ satisfies $K<\frac{1}{4}$.  In particular, (H3) is satisfied if
 ${\zeta^+} \le \log (23/3)\approx 2.0369$, in which case Corollary \ref{cor3.1} implies that $K$ is globally attractive for \eqref{eq:nicholson}, for all $N\in\N$. In contrast, for $N=2,3,\dots$ we have $a^+/a^->3/2$, and the result in \cite{FariaPrates} cannot be applied (conf. Remark \ref{rmk3.1}).
}\end{exmp}

 
\begin{exmp}\label{ex3} {\rm Consider a 1-periodic Nicholson's equation with two pairs of ``mixed delays":
\begin{equation}\label{exmpPer}
y'(t)=-\de(t)y(t)+p_1(t)y(t-\tau_1(t))e^{-y(t-\sigma_1(t))}+p_2(t)y(t-\tau_2(t))e^{-y(t-\sigma_2(t))},
\end{equation}
where  $ \de(t)=\eta_0(1+\frac{1}{2}\cos(2\pi t)), p_1(t)=\eta_1\left(1+\frac{1}{2}\cos(2\pi t)\right),p_2(t)=\eta_2\left(1+\frac{1}{2}\sin(2\pi t)\right)$ for some  $\eta_0, \eta_1>0,\eta_2\ge 0$,
and the delay functions
$\tau_j(t), \sigma_j(t)\, (j=1,2)$ are 1-periodic. Set $p(t)=p_1(t)+p_2(t)$.
  For $\al,\ga$ as in \eqref{alphagamma}, 
\begin{equation}\label {alfabeta2}\al=\min_{\R}\frac{p(t)}{\de(t)}=\frac{2\eta_1+(2-\sqrt 2)^2\eta_2}{2\eta_0},\ \ga=\max_{\R}\frac{p(t)}{\de(t)}=\frac{2\eta_1+(2+\sqrt 2)^2\eta_2}{2\eta_0}.
\end{equation}

Note that $\int_0^1 \de(s)\, ds={\eta_0}, \int_0^1p(s)\, ds=\eta_1+\eta_2$. Let $C_1^+(\R)$ be the space of continuous and 1-periodic functions $y:\R\to [0,\infty)$. As shown above,  equation \eqref{exmpPer}
  possesses at least one 1-periodic positive solution $y^*(t)$ in the cone $\mathcal{K}=\{y\in C_1^+(\R): y(t)\ge e^{-\eta_0}\|y\|_\infty\}$  if 
  \begin{equation}\label{expl_periodic}p(t)\ge e^{\eta_0}\de(t) \q {\rm and}\q
  p(t)\not\equiv e^{\eta_0}\de(t)\end{equation} as in   \eqref{periodic1}.      
   
    If $\eta_2=0$ and $ \eta_1>\eta_0$, \eqref{exmpPer} has the form \eqref{eq:nicholson} and in fact the   1-periodic positive solution $x^*(t)$ is the positive equilibrium $K=\log(\eta_1/\eta_0)$. According to the above notation, $\zeta^+\le \frac{3}{2}\eta_0 \bar \sigma$, where $\bar \sigma=\max_{j=1,2}\max_{t\in [0,1]}\sigma_j(t)$.
   From Corollary \ref{cor3.1} $K$ is a global attractor if 
  \begin{equation}\label{exp2.2}\log(\frac{\eta_1}{\eta_0})\big (e^{\frac{3}{2}\eta_0 \bar \sigma}-1\big)\le 1.\end{equation}
 
Now, let  $\eta_2>0$
  and assume that  \eqref{expl_periodic} holds. For instance,  this is the case   with $\eta_2>0$ and $ \eta_1\ge \eta_0e^{\eta_0}$, or with $ \eta_1< \eta_0e^{\eta_0}$ and $\eta_2\ge (1+\sqrt 2)(\eta_0e^{\eta_0}-\eta_1)$. To deduce the latter case, note that with $b=\eta_0e^{\eta_0}-\eta_1$ and $\eta_2\ge b>0$ we have
\begin{equation*}
\begin{split}
p(t)-e^{\eta_0}\de(t)&=-b\Big (1+\frac{1}{2}\cos(2\pi t))\Big)+\eta_2\Big (1+\frac{1}{2} \sin(2\pi t)\Big)\\
&\ge  \sqrt 2 b+ \frac{1}{2}b \big(\sin (2\pi t) -\cos (2\pi t)\big)+\frac{\sqrt 2}{2}b \sin  (2\pi t)\\
&\ge _{\not\equiv}  \sqrt 2 b- \frac{\sqrt 2}{2}b-\frac{\sqrt 2}{2}b=0
\end{split}
\end{equation*}
thus   \eqref{expl_periodic} is satisfied.
From the above estimates,  if
 $$ e^{\eta_0} (e^{\eta_0(1+ \frac{3}{2} \bar \sigma)}-1) \log \Big(\frac{2\eta_1+(2+\sqrt 2)^2\eta_2}{2\eta_0}\Big )\le 1,$$
 Theorem \ref{thm4.3} implies that the periodic solution $x^*(t)$ is a global attractor. When $\eta_2=0$, note that  \eqref{exp2.2}  is less restrictive than the above condition, which is not surprising since $x^*(t)\equiv K$.

 Next, we do not assume  the existence of any 1-periodic solution. Choose e.g. the delay functions
  $\tau_1(t)=0.1(1+\cos (2\pi t)),\ \tau_2(t)=0.2,\ \sigma_1(t)=0.2,\  \sigma_2(t)=0.1(1+\sin (2\pi t)).$
  Clearly, $\tau=\bar\tau=\bar \sigma =0.2$.
  With  $\al,\ga$ in \eqref{alfabeta2} and $P,D$ as in \eqref{alphagamma}, we have
  $
  D\le 0.3\eta_0, P\le  0.3(\eta_1+\eta_2).$
 With $\al>1$, Theorem \ref{thm2.1} asserts that \eqref{exmpPer} is permanent with all  positive solutions $x(t)$  satisfying \eqref{permGen} for 
  $$m=\log (\al )e^{-0.3(2\eta_0+\eta_1+\eta_2)},
  \ M=\log (\ga) e^{0.6(\eta_0+\eta_1+\eta_2)}$$
 and $\al,\ga$ as in \eqref{alfabeta2}. From Theorem \ref{thm3.3}.(b), all solution of \eqref{exmpPer}  are  globally attractive. if
 \begin{equation}\label{exp2.3}\log \Big(\frac{\eta_1+3\eta_2}{\eta_0}\Big )e^{0.6(\eta_0+x\eta_1+\eta_2)}
\left[ \frac{\log \Big(\frac{2\eta_1+(2+\sqrt 2)^2\eta_2}{2\eta_0}\Big )}{\log \Big( \frac{2\eta_1+(2-\sqrt 2)^2\eta_2}{2\eta_0}\Big )}e^{0.3(4\eta_0+3\eta_1+3\eta_2)}-1\right]<1.
\end{equation}
 }
\end{exmp}

\section{Conclusions and discussion}

In this paper, we give sufficient conditions for the permanence  of \eqref{general},  providing explicit positive lower and upper uniform bounds   for all solutions with initial condition in  $C_0^+$.
Under the additional condition (A5) on the size of the delays $\sigma_j(t)$, expressed in terms of an upper bound for the coefficient $\zeta^+$ in \eqref{zeta_M}, we show that   all positive solutions of \eqref{general} 
 are {\it globally attractive} (in $C_0^+$),  in the sense that $x(t)-y(t)\to 0$ as $t\to\infty$  for any two solutions  $x(t), y(t)$ with initial conditions in $ C_0^+$. Sharper results are obtained when there exists a positive equilibrium $K$ of \eqref{general} or when a particular solution is fixed (see  (K2) or (A4), respectively). Not only \eqref{general}  is much more general than  \eqref{eq:nicholson}, but also when applied to  \eqref{eq:nicholson}  the present results on the global stability of $K$   constitute a  significant improvement of those in \cite{FariaPrates}: 
we impose that $\be(t)$  satisfies (H1),~(H2), rather than being bounded from above and below by positive constants, and  the constraint $\frac{\max_j a_j}{\min_j a_j}<3/2$ in \cite{FariaPrates} is completely removed.  Comparison with further recent literature is also presented. On the other hand,   criteria for the existence of at least one positive periodic solution have been provided for {\it periodic} models \eqref{general}, see e.g.  \cite{Faria22,FO19}.  As a significant application, here we show how to use our results to obtain the global attractivity of such a positive periodic solution, providing an answer to an open problem in  \cite{Faria22}.


The method exploited here is based on the construction of an  associated DE $x_{n+1}=\mathfrak{h}(x_n)$ together with results concerning the existence of a globally attractive fixed point for DEs. A crucial point in this work was to construct a suitable $\mathfrak{h}$  (see \eqref{frakh}) which yields the desirable global attractivity under very mild constraints. 

 Certainly, this methodology has the potential to be applied to other models with mixed monotonicities given by functions $h_j(t,u,v)$ with linear growth in a vicinity of zero (such Mackey-Glass equations), in the sense that $h_j(t,u,v)$ is nondecreasing in $u$, nonincreasing in $v$ and $h_j(t,0,0)=0, \frac{\p}{\p x}h_j(t,0,0)=\frac{\p}{\p y}h_j(t,0,0)=1$. 
This was shown in the theoretical approaches developed in \cite{ER19,ElRuiz22}, yet they only apply to some models under other strong restrictions. It is not apparent however  how to extend it to equations with a sublinear growth, such as   the model studied in  \cite{huang20}.
In fact,    a variant of \eqref{N} is  the so-called ``neoclassical growth model"
 $x'(t)=- \delta x(t)+ p x^\ga(t-\tau)e^{-a x(t-\tau)}$ with $\ga\in (0,1)$,
  often used  in economics, where    the delay  accounts for a   reaction  to market changes. 
  In \cite{huang20}, Huang et al.  considered the  growth model obtained by replacing in \eqref{eq:nicholson} $ x(t-\tau_j(t)) $ by $ x^\ga(t-\tau_j(t)) \ (1\le j\le m)$, with $\ga\in (0,1)$, and gave a criterion for the global attractivity of its   positive equilibrium  $K_\ga$.
   It is not clear whether  difference equations can be used to derive the global atractivity of such model. 

In recent years, Nicholson-type systems have also been  intensively studied,
   in view of their applications as models for populations  structured in several  patches with migration of the populations among them, see e.g.  
\cite{BIT,FariaRost,Liu14,Xu2}.
A next future   project is to study the $n$-dimensional analogs of \eqref{eq:nicholson} and \eqref{general}. In the context of Nicholson system with patch structure and mixed monotonicities,
 Xu et al.  \cite{XuCaoGuo} gave already a criterion  for extinction of nonautonomous systems, whereas  El-Morshedy and Ruiz-Herrera  \cite{ElRuiz20} and Xu \cite{Xu2}  studied the global stability and attractivity of  autonomous models. Moreover, there are  some stability results in the literature for  periodic systems with a unique delay in each nonlinear term (see e.g.  \cite{Faria17,Faria21b,Liu14}  and references therein):  to extend them  to the broader scope  of mixed monotonicity Nicholson systems is also  a motivation for future research.

\end{document}